\documentclass[12pt]{amsart}

\usepackage{hyperref}
\usepackage{color}
\usepackage{pinlabel,enumerate}
\usepackage{graphicx}
\usepackage{xypic}
\usepackage{booktabs}
\usepackage{amsmath,amsfonts,amssymb}

\theoremstyle{plain}
\newtheorem*{theorem*}{Theorem}
\newtheorem*{lemma*} {Lemma}
\newtheorem*{corollary*} {Corollary}
\newtheorem*{proposition*}{Proposition}
\newtheorem*{conjecture*}{Conjecture}
\newtheorem{theorem}{Theorem}[section]
\newtheorem{lemma}[theorem]{Lemma}
\newtheorem*{theorem1*}{Theorem 1}
\newtheorem*{theorem2*}{Theorem 2}
\newtheorem*{theorem3*}{Theorem 3}

\newtheorem{proposition}[theorem]{Proposition}

\theoremstyle{remark}
\newtheorem*{remark}{Remark}

\newtheorem{example*}{Example}
\newtheorem*{claim}{Claim}

\theoremstyle{definition}

\def\G{\Gamma}

\def\op{\operatorname}
 \def\Q{\Bbb{Q}}  \def\Z{\Bbb{Z}}  
\def\N{\Bbb{N}}    
  \def\g{\gamma}  \def\bp{\begin{pmatrix}}
\def\sm{\setminus} \def\ep{\end{pmatrix}} \def\bn{\begin{enumerate}} 
   \def\en{\end{enumerate}}
\def\ba{\begin{array}} \def\ea{\end{array}} 
\def\int{\op{int}}     \def\ti{\widetilde}

\def\ker{\mbox{Ker}}\def\be{\begin{equation}} \def\ee{\end{equation}} 
   
 \def\hom{\mbox{Hom}}

\def\ol{\overline}

\def\wti{\widetilde}
\def\OO{\mathcal{O}}

\begin{document}
\title{The virtual fibering theorem for $3$-manifolds}
\author{Stefan Friedl}
\address{Fakult\"at f\"ur Mathematik\\ Universit\"at Regensburg\\   Germany}
\email{sfriedl@gmail.com}
\author{Takahiro Kitayama}
\address{Department of Mathematics, Tokyo Institute of Technology,
2-12-1 Ookayama, Meguro-ku, Tokyo 152-8551, Japan}
\email{kitayama@math.titech.ac.jp}

\date{\today}

\begin{abstract}
In 2007 Agol showed that if $N$ is an aspherical compact 3-manifold with empty or toroidal boundary such that $\pi_1(N)$ is virtually RFRS,
then $N$ is virtually fibered. We give a largely self-contained proof of Agol's  theorem  using complexities of sutured manifolds.
\end{abstract}

\maketitle

\section{Introduction}\label{section:rfrs}

In 1982 Thurston \cite[Question~18]{Th82} asked whether every  hyperbolic 3-manifold
is virtually fibered, i.e. whether every hyperbolic 3-manifold admits a finite cover which fibers over $S^1$.

Evidence towards an affirmative answer was given by many authors, including Agol--Boyer--Zhang \cite{ABZ08}, Aitchison--Rubinstein \cite{AR99}, Button \cite{Bu05}, DeBlois \cite{DeB10},
 Gabai \cite{Ga86}, Guo--Zhang \cite{GZ09},  Leininger \cite{Lei02}, Reid \cite{Re95}  and  Walsh \cite{Wa05}.

The first general virtual fibering theorem was proved by Agol in 2007.
In order to state the theorem we need two more definitions:
\bn
\item  a group $\pi$ is  \emph{residually finite rationally solvable} or \emph{RFRS} if
$\pi$ admits a filtration   $\pi=\pi_0\supset \pi_1 \supset \pi_2\dots $
such that the following hold:
\bn
\item $\cap_k \pi_k=\{1\}$,
\item for any $k$ the group $\pi_k$ is a normal, finite index subgroup of $\pi$,
\item for any $k$ the map $\pi_k\to \pi_k/\pi_{k+1}$ factors through $\pi_k\to H_1(\pi_k;\Z)/\mbox{torsion}$.
\en
\item Given a $3$-manifold $N$, we say $\phi\in H^1(N;\Q)=\hom(\pi_1(N),\Q)$ is \emph{fibered}
if there exists an $n\in \N$ and a locally trivial fiber bundle $p\colon N\to S^1$ such that $\phi=\frac{1}{n}\cdot p_*\colon \pi_1(N)\to \Q$.
\en
We can now state Agol's \cite[Theorem~5.1]{Ag08} theorem.

\begin{theorem}\label{thm:ag08intro}\textbf{\emph{(Agol)}}\label{thm:51}
Let $N$ be an irreducible   $3$--manifold with empty or toroidal boundary such that $\pi_1(N)$ is  virtually RFRS.
Let $\phi\in H^1(N;\Q)$ be non--trivial. Then  there exists a finite cover $q\colon \ti{N}\to N$
such that  $q^*\phi$ is the limit of a sequence of  fibered classes in $H^1(\wti{N};\Q)$.
\end{theorem}

The key idea in the proof of the theorem is that the RFRS condition ensures that given a Thurston norm minimizing surface
one can find `enough' surfaces in finite covers to `reduce the complexity of  the guts' by perturbing the initial surface appropriately.
Agol uses the theory of `least-weight taut normal surfaces' introduced and developed by
Oertel \cite{Oe86} and Tollefson and Wang \cite{TW96} to carry through this program.

In the introduction to \cite{Ag08} Agol writes that `the natural setting [\dots] lies in sutured manifold hierarchies'.
We pick up this suggestion and provide a proof of Theorem \ref{thm:ag08intro}  using sutured manifolds and their hierarchies.
In our proof we only use standard results about the Thurston norm and  sutured manifold decompositions
(see \cite{Th86,Ga83}) and a complexity for sutured manifolds defined by Gabai \cite{Ga83}.
At the core our argument also follows the above `key idea', but for the most part the treatment of the argument is somewhat different
from Agol's original proof.
\medskip

In a stunning turn of events it has been shown over the last few years that  most $3$-manifold groups are in fact virtually RFRS.
More precisely, the following theorem was proved by Agol \cite{Ag13} and Wise \cite{Wi12}
in the hyperbolic case and by Przytycki-Wise \cite{PW12} in the case of a $3$-manifold with a non-trivial JSJ decomposition.

\begin{theorem}\textbf{\emph{(Agol, Przytycki-Wise, Wise)}}\label{thm:apw}
If $N$ is an irreducible  $3$-manifold with empty or toroidal boundary which is not a graph manifold, then $\pi_1(N)$ is virtually RFRS.
\end{theorem}

Furthermore it follows from work of Liu \cite{Li11} and Przytycki-Wise \cite{PW11} that the fundamental group of a graph manifold with boundary
is also virtually RFRS. Finally Liu \cite{Li11} showed that the fundamental group of a closed aspherical graph manifold
is virtually RFRS if and only if $N$ is non-positively curved, i.e. if it admits a Riemannian metric of non-positive curvature.
Combining these results with Theorem \ref{thm:51} we thus obtain the following result.

\begin{theorem}\label{thm:virtfib}\label{thm:apwfibintro}
Let $N$ be an irreducible   $3$--manifold with empty or toroidal boundary and let $\phi\in H^1(N;\Q)$ be non--trivial.
Suppose one of the following two conditions hold:
\bn
\item $N$ is not a closed graph manifold, or
\item $N$ is a closed graph manifold which is non-positively curved,
\en
then  there exists a finite cover $q\colon \ti{N}\to N$ such that  $q^*\phi$ is the limit of a sequence of  fibered classes in $H^1(\wti{N};\Q)$.
\end{theorem}

\begin{remark}
\bn
\item If $\pi_1(N)$ is infinite and virtually RFRS, then there exists a finite cover with positive first Betti number.
It therefore follows from Theorem \ref{thm:virtfib} that if $N$ is an irreducible  $3$-manifold with empty or toroidal boundary which is not a graph manifold, then $N$ is virtually fibered. In particular Theorem \ref{thm:virtfib} gives an affirmative answer to
Thurston's question.
\item The work of Agol \cite{Ag13}, Przytycki-Wise \cite{PW12} and Wise \cite{Wi12} resolves not only Thurston's Virtual Fibering Conjecture but also has a long list of other consequences. We refer to the survey paper \cite{AFW12} for a comprehensive discussion.
\item Let $N$ be an aspherical   $3$--manifold. If $N$  is not a closed graph manifold, then it follows from
work of Leeb \cite{Leb95} (see also \cite{Br99,Br01}) that $N$ is non-positively curved.
Combining this with the above results we see that an aspherical $3$-manifold $N$ is non-positively curved if and only if $\pi_1(N)$ is virtually RFRS.
\item There are graph manifolds which are virtually fibered but whose fundamental groups are not virtually RFRS.
One class of such graph manifolds is given by Sol-manifolds (see e.g. \cite{Ag13,AFW12}).
\item It follows from classical arguments that the conclusion of Theorem \ref{thm:virtfib} in fact holds for any virtually fibered graph manifold.
\item The conclusion that any cohomology class can be approximated by fibered classes in a suitable finite cover
 has been crucial in the applications to twisted Alexander polynomials and to the study of surfaces of minimal complexity in $4$-manifolds with a free $S^1$-action
 (see  \cite{FV12,FV14}).
\en
\end{remark}

For completeness' sake we also mention Agol's result on finite covers of taut sutured manifolds,
even though it plays no role in the later part of the paper.
Using the argument in the proof of Theorem \ref{thm:51} and using a `doubling' argument Agol proves  that given any taut sutured manifold with virtually RFRS fundamental group
there exists a finite cover which admits  a depth one taut oriented foliation. (We will not define these notions  and we refer instead to  \cite{Ga83,Ag08,CC03} for background information and precise definitions.)
More precisely, Agol  \cite[Theorem 6.1]{Ag08} proved the following result.

\begin{theorem}\textbf{\emph{(Agol)}}\label{thm:61}
Let $(N,\g)$ be a  taut sutured manifold such that $\pi_1(N)$ is virtually RFRS.
Then there exists a finite covering $p\colon (\ti{N},\ti{\g})\to (N,\g)$ such that $(\ti{N},\ti{\g})$ admits a depth one taut oriented foliation.
\end{theorem}

In the above discussion we already saw that the fundamental group of any irreducible $3$-manifold with non-trivial toroidal boundary is virtually RFRS. A straightforward doubling argument (see e.g. \cite[Section~5.3]{AFW12}) shows that in fact the fundamental group of any irreducible $3$-manifold with non-empty incompressible boundary is virtually RFRS.
Combining this observation with Theorem \ref{thm:61} we obtain  the following theorem.

\begin{theorem}
Let $(N,\g)$ be a  taut-sutured manifold.
Then there exists a finite covering $p\colon (\ti{N},\ti{\g})\to (N,\g)$ such that $(\ti{N},\ti{\g})$ admits a depth one taut-oriented foliation.
\end{theorem}

The paper is organized as follows.
In Sections \ref{section:thurstonnorm} and \ref{section:complexities} we recall some standard facts about the Thurston norm and sutured manifolds.
Along the way we will also make some preparations for the proof of Proposition \ref{prop:perturbguts1}.
This proposition allows us to carefully arrange surfaces to `cut the guts' of a given surface.
This result  is the technical heart of the paper and we  give a very detailed proof of it.
In Proposition \ref{prop:cutguts} we then summarize the effect of `cutting by a surface' on the complexities of the guts of a given surface.
Finally in the last section we present our proof of Theorem \ref{thm:ag08intro}.

\subsection*{Convention.}
All manifolds are assumed to be compact and oriented. We do not assume that spaces are connected,
nonetheless, if we talk about the fundamental group of a space without specifying a base point, then we implicitly assume that the space is connected. All surfaces in a $3$-manifold are assumed to be properly embedded, unless we say explicitly otherwise.
If $N$ is a $3$-manifold and  $R\subset N$ a properly embedded surface and $a>0$,
then we denote by $R\times [-a,a]$ a  neighborhood of $R$ such that
$(R\times [-a,a]) \cap \partial N=\partial R\times [-a,a]$.
Finally, given a submanifold $S\subset M$  we denote by $\nu S$ an open tubular neighborhood around $S$.

\subsection*{Acknowledgment.} This work was supported by Grant-in-Aid for JSPS Fellows.
We wish to thank Ian Agol and Steve Boyer for helpful conversations. 
We are also grateful to the referee for many helpful comments.
Finally, we are especially indebted to Andr\'as Juh\'asz for pointing out a mistake in our first version of this paper.

\section{The Thurston norm}\label{section:thurstonnorm}

\subsection{The Thurston norm and fibered classes}

Let $S$ be a surface with connected components $S_1\cup\dots \cup S_k$.
We then refer to
\[ \chi_-(S):=\sum_{i=1}^k \max\{-\chi(S_i),0\}\]
as  the \emph{complexity} of $S$. Now let $N$ be a 3-manifold  and let $\phi\in H^1(N;\Z)$.
It is well--known that  any class in $H^1(N;\Z)$ is dual to a properly embedded surface.
The   \emph{Thurston norm of $\phi$} is defined as
\[ x_N(\phi):= \min\{\chi_-(S)\, |\, S \subset N\mbox{ properly embedded and dual to }\phi\}.\]
We will drop the subscript `$N$', when the manifold $N$ is understood.

Thurston \cite{Th86} showed that  $x$ is a seminorm on $H^1(N;\Z)$, which implies that $x$  can be extended to a seminorm on $H^1(N;\Q)$.
We denote the seminorm on $H^1(N;\Q)$ also by  $x$.
Throughout the paper we will freely go back and forth between $H^1(N;\Q)$ and $H_2(N,\partial N;\Q)$.
In particular we will consider the Thurston norm  also for classes in $H_2(N,\partial N;\Q)$.

 Thurston furthermore proved that  the Thurston norm ball
\[ B(N):=\{ \phi \in H^1(N;\Q)\, | \, x(\phi)\leq 1 \}\]
 is a (possibly non--compact) finite convex polytope.
 A \emph{Thurston cone of $N$} is defined to be either an open cone $\{rf\,|\,r>0,f\in F\}$ on a face $F$ of $B(N)$
 or a maximal connected subset of $H^1(N;\Q)\sm \{0\}$ on which $x$ vanishes.
 The Thurston cones have the following properties:
\bn
\item if $\phi,\psi$ lie in a  Thurston cone $C$, then $\phi+\psi\in C$ and given any  $r>0$ we have  $r\phi\in C$,
\item the Thurston cones are disjoint and their union  equals  $H^1(N;\Q)\sm \{0\}$,
\item the Thurston norm is additive precisely on the closures of Thurston cones, i.e. given $\phi,\psi\in H^1(N;\Q)$ we have
\[ x(\phi+\psi)=x(\phi)+x(\psi) \Leftrightarrow \mbox{\,there\,exists\,a\,Thurston\,cone\,$C$\,with\,}\phi,\psi\in \ol{C}.\]
\en

In the following we say that an integral class $\phi \in H^1(N;\Z)=\hom(\pi_1(N),\Z)$ is  \emph{fibered}  if  there exists a fibration $p\colon N\to S^1$ such that
$\phi=p_*\colon \pi_1(N)\to \Z$. We say $\phi\in H^1(N;\Q)$ is \emph{fibered} if a non-trivial integral multiple of $\phi$ is fibered.
Thurston \cite{Th86} showed that the set of fibered classes equals the union of some
 top--dimensional Thurston cones. These cones are referred to as the \emph{fibered cones} of $N$.

\subsection{Subordination}

Given  two non-zero cohomology classes $\phi,\psi\in H^1(N;\Q)$
we say \emph{$\phi$ is subordinate to $\psi$} if  $\phi\in \ol{C}$ where $C$ is the  unique Thurston cone which contains $\psi$.
We collect several properties of subordination in a lemma:

\begin{lemma}\label{lem:subordinate}
\bn
\item Subordination is transitive, i.e. if $\phi$ is subordinate to $\psi$ and $\psi$ is subordinate to $\varphi$,
then $\phi$ is subordinate to $\varphi$.
\item Given any two non-zero cohomology classes $\phi,\psi \in H^1(N;\Q)$ there exists an $m\in \N$ such that $\phi$ is subordinate to $m\phi+\psi$.
\item If  $\phi$ is subordinate to $m\phi+\psi$ for some $m$, then $\phi$ is also subordinate to $k\phi+\psi$ for any $k\geq m$.
\item  Let $p\colon \ti{N}\to N$ be a finite cover and let  $\phi,\psi\in H^1(N;\Q)$ be  two non-zero cohomology classes.
Then  $\phi$ is subordinate to $\psi$ if and only if $p^*\phi$ is subordinate to $p^*\psi$.
\en
\end{lemma}

The first three statements are straightforward to verify.
The last statement is an immediate consequence of the fact that given any
cover  $p\colon \ti{N}\to N$ of degree $k$ and any $\phi\in H^1(N;\Q)$
we have  $x(p^*\phi)=k\cdot x(\phi)$ (see \cite[Corollary~6.13]{Ga83}).
Put differently, $p^*\colon H^1(N;\Q)\to H^1(\ti{N};\Q)$ is up to the scaling factor $k$ an isometry of vector spaces.

\section{Complexities for  sutured manifolds}\label{section:complexities}

\subsection{Sutured manifolds}

A sutured manifold $(M,R_-,R_+,\gamma)$ consists of a $3$-manifold $M$
 together with a decomposition of its boundary
\[ \partial M=-R_-\cup \g \cup R_+\]
into oriented submanifolds where the following conditions hold:
\bn
\item $\g$ is a disjoint union of annuli.
\item $R_-$ and $R_+$ are disjoint.
\item If $A$ is a component of $\g$, then $R_-\cap A$ is a  boundary component of $A$ and  of $R_-$, and similarly for $R_+\cap A$. Furthermore, $[R_+\cap A]=[R_-\cap A]\in H_1(A;\Z)$ where we endow $R_\pm \cap A$ with the orientation coming from the boundary of $R_\pm$.
\en
(Here we give $\partial M$ the orientation such  that  $R_+$ are precisely those components of $\overline{\partial M\sm \g}$
 whose normal vectors point out of  $M$.)

We sometimes just write $(M,\gamma)$ instead of $(M,R_-,R_+,\gamma)$, but it is important to remember that $R_-$ and $R_+$ are part of the structure of a sutured manifold.

Finally a simple example of a sutured manifold is given as follows: Let $R$ be a  surface, then
 \[(R\times [-1,1],R\times -1,R\times 1,\partial R\times [-1,1])\]
  is a sutured manifold.
 We refer to such a sutured manifold as a \emph{product sutured manifold}.

\subsection{Taut sutured manifolds and Thurston norm minimizing surfaces}
A sutured manifold $(M,R_-,R_+,\gamma)$ is called  \emph{taut} if $M$ is irreducible and if $R_-$ and $R_+$
have minimal complexity among all surfaces representing $[R_-]=[R_+]\in H_2(M,\g;\Z)$.

Let $R$ be a surface in a closed $3$-manifold $N$.
 We say that $R$ is \emph{good} if $R$ has no spherical components and no component which bounds a  solid torus.
  Furthermore we say $R$ is \emph{Thurston norm minimizing} if $R$ has minimal complexity in its homology class $[R]\in H_2(N,\partial N;\Z)$.
It is clear  that  any homology class can be represented by a good Thurston norm minimizing surface.

Note that if $R$ is a good Thurston norm minimizing surface in an irreducible $3$-manifold, then a standard argument using the Loop Theorem
(see \cite[Chapter~4]{He76}) shows that $R$ is also $\pi_1$-injective.

To a  surface $R$ in a closed $3$-manifold $N$ we now associate  the sutured manifold
\[ N(R) = (N\sm R\times (-1,1),R\times 1,R\times -1,\emptyset).\]
We conclude this section with the following two observations regarding $N(R)$:
\bn
\item If $N$ is irreducible and if  $R$ is a Thurston norm minimizing surface without spherical components, then  $N(R)$ is a taut sutured manifold.
\item The surface $R$ is a fiber of a fibration $N\to S^1$ if and only if $N(R)$ is a product sutured manifold.
\en

\subsection{Sutured manifold decompositions}

We now recall the definition of a sutured manifold decomposition which also goes back to  Gabai \cite{Ga83}.
Let $(M,R_-,R_+,\g)$ be a sutured manifold.
 We say that a  properly embedded surface  $S$ is a \emph{decomposition surface} if the following condition holds:
  for any component $A$  of $\g$ every component of $A\cap S$ is either  a  non-separating arc in $A$, or
it is a closed  curve which is homologous to  $[A\cap R_-]=[A\cap R_+]\in H_1(A;\Z)$.

Let $S$ be a decomposition surface  of $(M,R_-,R_+,\g)$.  Gabai \cite[Definition~3.1]{Ga83} defines the sutured manifold decomposition
\[ (M,R_-,R_+,\g)\overset{S}{\rightsquigarrow}(M',R_-',R_+',\g')\]
where
\[ \ba{rcl} M'&=&M\sm S\times (-1,1),\\
\g'&=&(\g\cap M')\cup \ol{\nu(S'_+\cap R_-)}\cup \ol{\nu (S_-'\cap R_+)},\\
R_+'&=&((R_+\cap M')\cup S_+')\sm \int \gamma'\\
R_-'&=&((R_-\cap M')\cup S_-')\sm \int \gamma'.\ea \]
Here $S_+'$ (respectively $S_-'$) is the union of the components of $(S\times -1\cup S\times 1)\cap M'$ whose normal vector points out of (respectively into) $M'$. Furthermore, by $\nu(S'_{\pm}\cap R_{\mp})$ we mean an open tubular neighborhood of $S'_{\pm}\cap R_{\mp}$ in $\partial M$.
We say that a decomposition surface $S$ is \emph{taut} if all the components of the  sutured manifold decomposition along $S$ are taut.

We make the following  observations:
\bn
\item If $\g=\emptyset$, then any surface in $M$ is a decomposition surface for $(M,\g)$.
\item If each component of $S$ is a $\pi_1$-injective surface, then for any component of $M'$ the inclusion into $M$ induces a monomorphism of fundamental groups.
\item If $N$ is a closed $3$-manifold  and if $R\subset N$ is a closed  surface, then
$R$ is a decomposition surface for the sutured manifold $(N,\emptyset,\emptyset,\emptyset)$, and $N(R)$ is precisely the result of the decomposition along $R$.
\item If $(M,\g)$ is a sutured manifold and if $S\subset M$ is a decomposition surface which is boundary parallel,
then the resulting sutured manifolds $(M',\g')$ is a union  of product sutured manifolds and a sutured manifold $(M'_0,\g'_0)$ which
is canonically diffeomorphic to $(M,\g)$.
\en

\subsection{Guts of a sutured manifold}

Let $(M,R_-,R_+,\g)$ be a taut sutured manifold.
An \emph{admissible annulus} is an  annulus  $S$ in $M$ which does not cobound a solid cylinder in $M$ and  such that one boundary component of $S$ lies on $R_-$ and the other one lies on $R_+$. Furthermore, an \emph{admissible disk} 
is a disk $S$ in $M$  such that $S\cap R_-$ and $S\cap R_+$ consist of an interval each.

We have the following elementary but very useful lemma  (see also  \cite[Lemma~3.12]{Ga83}).

\begin{lemma}\label{lem:ga83lemma312}
Let $(M,R_-,R_+,\g)$ be a taut sutured manifold. 
Then any admissible annulus and any admissible disk is a taut decomposition surface.
\end{lemma}

An \emph{admissible decomposition surface} for a sutured manifold  $(M,R_-,R_+,\g)$ is a disjoint union of admissible annuli and disks in  $(M,R_-,R_+,\g)$. 
Given such an $S$ we can perform the sutured manifold decomposition
\[ (M,R_-,R_+,\g)\overset{S}{\rightsquigarrow}(M',R_-',R_+',\g').\]
We refer to any component of $M'$ which is a product sutured manifold as a \emph{window of $(M,R_-,R_+,\g)$}
and we refer to any component of $M'$ which is not a product sutured manifold  as a \emph{gut of $(M,R_-,R_+,\g)$}.
Note that the definition of window and gut depends on the choice of the admissible decomposition surface.
Nonetheless, from the context it is usually clear what admissible decomposition surface we are working with and we will therefore leave the dependence on $S$ unmentioned.

\begin{lemma}\label{lem:guts}
Let $(M,R_-,R_+,\g)$ be a sutured manifold such that  $M$ is  irreducible. We pick an admissible decomposition surface.  Then the following hold:
\bn
\item The guts and windows are $\pi_1$-injective submanifolds of $M$.
\item The fundamental group of a gut is non-trivial.
\item If  $(M,R_-,R_+,\g)$ is taut, then the windows and guts are also taut.
\en
\end{lemma}

The first statement follows from the observation that the components of an admissible decomposition surface  are $\pi_1$-injective
if $M$ is irreducible, the second statement is a consequence of the irreducibility of $M$
(or alternatively of the Poincar\'e conjecture) and the third statement is a consequence of Lemma \ref{lem:ga83lemma312}.
\medskip

We conclude this section with the  following proposition.

\begin{proposition}\label{prop:pullbackguts}
Let $(M,R_-,R_+,\g)$ be a taut sutured manifold and let $p\colon (\ti{M},\ti{\g})\to (M,\g)$ be a finite cover.
\bn
\item If $(M,\g)$ is taut, then $(\ti{M},\ti{\g})$ is also taut.
\item If  $S\subset M$ is an admissible decomposition surface,
then $p^{-1}(S)$ is an admissible decomposition surface for $\ti{M}$, and  the windows and guts of
$(\ti{M},\ti{\g})$ are precisely the preimages  of the windows and guts of $(M,\g)$.
\en
\end{proposition}

\begin{proof}
Let $(M,R_-,R_+,\g)$ be a taut sutured manifold
and let $p\colon (\ti{M},\ti{\g})\to (M,\g)$ be a finite cover.

We first suppose that $(M,\g)$ is taut. It follows from the Equivariant Sphere Theorem, see \cite[p.~647]{MSY82},
and  work of Gabai (e.g. by combining Corollaries 5.3 and 6.13 and Lemma~6.14 of \cite{Ga83} with
Corollary 2 of \cite{Th86}, that $(\ti{M},\ti{\g})$ is also taut.

Now let $S\subset M$ be an admissible decomposition surface.
Let $G=(G,S_-,S_+)$ be a gut of $M$ and let $\wti{G}=(\wti{G},\wti{S}_-,\wti{S}_+)$
be a component of $p^{-1}(G)$. 
We have to show that $\wti{G}=(\wti{G},\wti{S}_-,\wti{S}_+)$ is not a product sutured manifold. 
Since $G=(G,S_-,S_+)$ is not a product sutured manifold it follows from \cite[Theorem~10.5]{He76}
that precisely one of the following two cases can occur:
\bn
\item  $G$ is the twisted $I$-bundle over a Klein bottle and $S_-=\partial G$, or 
\item $\pi_1(S_-)$ has infinite index in $\pi_1(G)$.
\en
We now consider these two cases separately:
\bn
\item If $G=(G,S_-,S_+)$ is a twisted $I$-bundle over a Klein bottle with $S_-=\partial G$ and $S_+=\emptyset$, then $\wti{G}=(\wti{G},\wti{S}_-,\wti{S}_+)$ is a sutured manifold with $\wti{S}_+=\emptyset$, i.e. $\wti{G}$ is not a product sutured manifold.
\item If  $\pi_1(S_-)$ has infinite index in $\pi_1(G)$, then  $\pi_1(\wti{S}_-)$ also has infinite index in $\pi_1(\wti{G})$, which implies that $(\wti{G},\wti{S}_-,\wti{S}_+)$ is not a product sutured manifold.
\en
\end{proof}

\subsection{The double-curve sum of surfaces}
Let $N$ be a closed  $3$-manifold
and let $R$ and $F$ be two embedded surfaces which are in general position.
Note that by the standard `cut and paste' technique applied to the  intersection curves  of $R$ and $F$
we can turn the immersed surface $R\cup F$ into an  embedded surface $R\uplus F$.
 The surface $R\uplus F$  is sometimes called the double-curve sum of $R$ and $F$.
Note that $R\uplus F$  represents the same homology class as $R\cup F$ and that furthermore $R\uplus F$ has the same complexity as $R\cup F$.

Now let $R$ and $F$ be  two properly embedded surfaces  in $N$ in general position.
\bn
\item A \emph{filling ball for $(R,F)$} is an embedded ball $B\subset N$ such that $\partial B\subset R\cup F$
 as oriented surfaces.
\item A \emph{filling solid torus for $(R,F)$} is an embedded solid torus  $X\subset N$  such that $\partial X\subset R\cup F$  as oriented surfaces.
\en
(Here we view $B$ and $X$ as oriented manifolds where the orientation does not necessarily have to agree with the orientation of $N$.)
We then say that $R$ and $F$ form a \emph{good pair} if there are no filling  balls and no filling  solid tori for $(R,F)$.

We will later on make use of the following elementary  lemma:

\begin{lemma}\label{lem:doublecurvesum}
Let $N$ be a closed irreducible $3$-manifold and let $R$ and $F$ be a good pair of  embedded surfaces in $N$.
Then the following hold:
\bn
\item $R$ and $F$ are good,
\item $R\uplus F$ is good,
\item $F\cap N(R)$ is a decomposition surface for $N(R)$,
\item
there exist decomposition annuli $C_1,\dots,C_k$ which are in one-to-one correspondence with the components of $R\cap F$
such that the following diagram commutes:
\[ \xymatrix{ N \ar@{~>}[rr]^{R}\ar@{~>}[d]^{R\uplus F}  &&  N(R) \ar@{~>}[d]^{F\cap N(R)} \\
N(R\uplus F)\ar@{~>}[rr]^{C_1\cup\dots\cup C_k} && (M,\g).}\]
\en
\end{lemma}

A schematic illustration for $R\uplus F$ and the decomposition annuli $C_i$ is given in Figure \ref{fig:rupluss}.
\begin{figure}
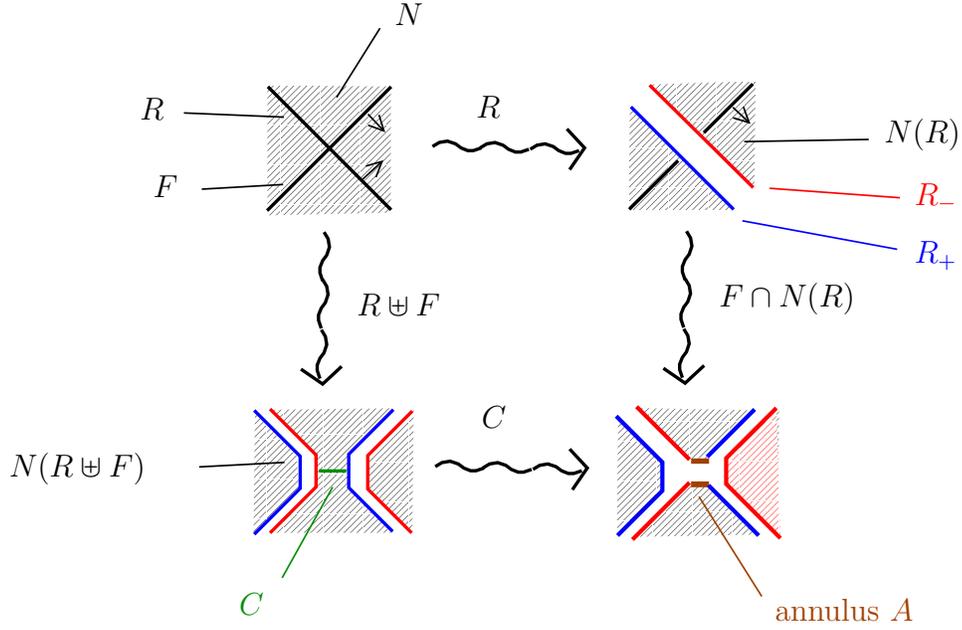
\caption{Schematic picture for decomposing along $R$ and $F$.}\label{fig:rupluss}
\end{figure}

\subsection{Complexity of sutured manifolds}

Gabai \cite[Definition~4.11]{Ga83}
associates to each connected sutured manifold $(M,R_-,R_+,\g)$ an invariant  $c(M,R_-,R_+,\g)\in \OO$
which we refer to as the \emph{complexity}
of $(M,R_-,R_+,\g)$. Here $\OO$ is a totally ordered set with the property that any strictly descending chain in $\OO$ starting at a given element is finite. We denote the  minimal element of $\OO$  by $0$. We refer to \cite[Definition~4.3]{Ga83} for details.
\footnote{Note that our notation and name differ from \cite{Ga83}: Gabai denotes this invariant $\ol{C}(M,R_-,R_+,\gamma)$ and calls it the `reduced complexity'.}

Gabai \cite[Section~4]{Ga83} proved the following theorem.

\begin{theorem}\label{thm:complexityfunction}
Let $(M,\g)$ be a connected sutured manifold and let $(M,\g)\overset{S}{\rightsquigarrow} (M',\g')$ be a   sutured manifold decomposition along
 a    connected decomposition surface $S$.
Suppose that $(M,\g)$ and $(M',\g')$ are taut.  Let $(M'_0,\g'_0)$ be a component
of $(M',\g')$. Then
\[ c(M'_0,\g'_0)\leq c(M,\g).\]
Furthermore, if $S$ is not boundary parallel, e.g. if $[S]$ is non-trivial in $H_2(M,\partial M;\Z)$, then
\[ c(M'_0,\g'_0)<c(M,\g).\]
\end{theorem}

\begin{remark}
\bn
\item We could also use the complexity $\hat{C}(M,R_-,R_+,\gamma)$ introduced by Scharlemann \cite[Definition~4.12]{Sc89}.
It follows from \cite[Definition~4.12,~Remark~4.13~(b)~and~Theorem~4.17]{Sc89} that
the conclusion of Theorem \ref{thm:complexityfunction} holds in an analogous way for Scharlemann's complexity.
\item Juh\'asz \cite{Ju06,Ju08} defines and studies in detail the `sutured Floer homology $SFH(M,\g)$' for `balanced' sutured manifolds.
The total rank of $SFH(M,\g)$ is a very useful complexity for balanced sutured manifolds and it has properties similar
to  Theorem \ref{thm:complexityfunction}. It would be interesting to give another proof of  Theorem \ref{thm:ag08}  using $SFH(M,\g)$. This though would require some adjustments since not all sutured manifolds which occur in our proof are balanced (e.g. if $(M,R_-,R_+,\g)$ is balanced, then $R_-$ and $R_+$ have no  closed components).
\en
\end{remark}

\section{Perturbations of homology classes}

The key to proving the Virtual Fibering Theorem is to show that given a
good Thurston norm minimizing surface $R$ and a homology class $\psi\in H_2(N;\Z)$
one can find a surface $F$ such that given any gut or window $X$ of $N(R)$
the intersection $F\cap X$ is a taut decomposition surface for $X$ which represents the same class as the restriction
of $\psi$ to $H_2(X,\partial X;\Z)$.

We start out with the following proposition.

\begin{proposition}\label{prop:perturbguts1}
Let $N$ be a closed irreducible connected $3$-manifold  and let $R$ be a   good Thurston norm minimizing surface.
Then for any choice of admissible decomposition surface for $N\sm R\times (-4,4)$ and any choice of  $\psi\in H_2(N;\Z)$ there exists an $m\in \N$ and a  surface $F$
with the following properties:
\bn
\item[(W1)] $[R]$ is subordinate to $m[R]+\psi$ and  $F$  represents $m[R]+\psi$,
\item[(W2)] $F\uplus (R\times -3\cup R\times 3)$ is Thurston norm minimizing,
\item[(W3)] the intersections $F\cap R\times [-4,-2]$ and $F\cap R\times [2,4]$ are product surfaces,
\item[(W4)] if $X$ is a gut or a window of $N\sm R\times (-4,4)$, then $F\cap X$ is a  decomposition surface,
\item[(W5)] $F$ and $R\times -3\,\,\cup\,\, R\times 3$ are a good pair.
\en
\end{proposition}

In the proposition we implicitly identified a tubular neighborhood of $R$ in $N$ with $R\times [-4,4]$.
Strictly speaking we should write $R\times \{-3\}$ and $R\times \{3\}$, but in our opinion
 $F\uplus (R\times \{-3\}\cup R\times \{3\})$ is less readable than  $F\uplus (R\times -3\cup R\times 3)$.

This proposition is the technical heart of our proof of the Virtual Fibering Theorem and we therefore give a detailed proof of the proposition.
A very schematic picture for Proposition \ref{prop:perturbguts1} is given in Figure \ref{fig:cutsurface}.
\begin{figure}
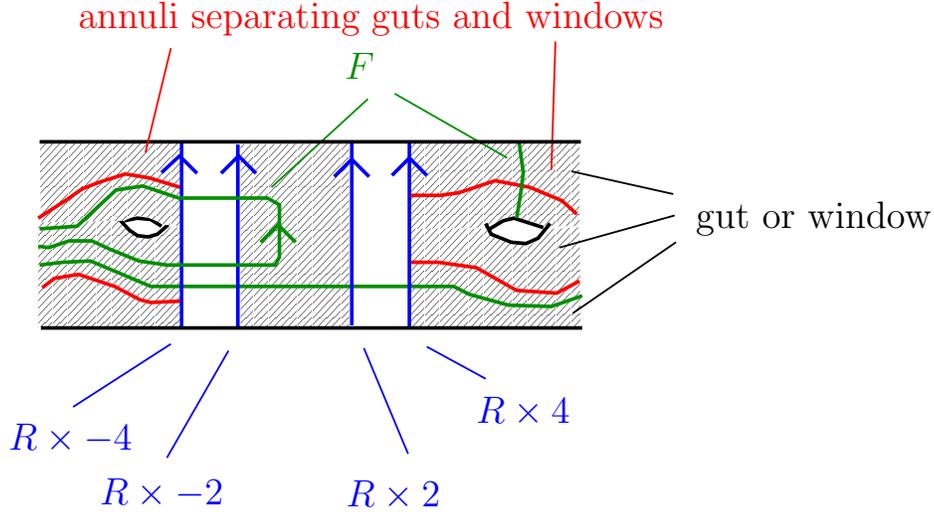
\caption{Schematic picture for Proposition \ref{prop:perturbguts1}.}\label{fig:cutsurface}
\end{figure}

\begin{proof}
Let $N$ be a closed irreducible $3$-manifold
and let $R$ be a   good Thurston norm minimizing surface. We pick a tubular neighborhood $R\times [-5,5]$ for $R$.
We write ${M}=N\sm R\times (-1,1)$ which we view as a sutured manifold $({M},{\g})$ in the usual way.

We pick an admissible decomposition surface for $M$. We  denote by $({M}_i,R_{i+},R_{i-},\g_i)$, $i=1,\dots,r$ the
corresponding guts and windows of  ${M}$.
Finally we denote by ${A}_1,\dots,{A}_s$ the collection of all the  components of the $\g_i$.
Note that we can and will assume that for each ${A}_i$ the intersection with $R\times [-5,-1]$ and  $R\times [1,5]$ consists
of a union of product annuli.

Before we state the first claim of the proof we need to introduce one more definition.
Let $S$ be a surface and  let $c$ be a component of $S\cap {A}_i$ which is a closed  curve.
We pick a $j$ such that ${A}_i$ is a component of $\g_j$, i.e. such that $A_i$ lies on $\partial {M}_j$.
Note that $c$ is a component of $\partial (S\cap {M}_j)$ and it thus inherits an orientation.
We now say that $c$ is \emph{positive} if $[c]=[R_{j\pm}\cap A_i]\in H_1({A}_i;\Z)$
and we say $c$ is \emph{negative} if $[c]=-[R_{j\pm}\cap A_i]$.
It is straightforward to see that if we chose the other  ${M}_k$ for which ${A}_i$ is a component of $\g_k$, then
 the orientation of $c$ flips and $[R_{k\pm}\cap A_i]=-[R_{j\pm}\cap A_i]$, which implies that we would get the same sign.

We can now formulate our first claim.

\begin{claim}
Let $\psi\in H_2(N;\Z)$.
There exists an $l\in \N$  and a surface $D$
with the following properties:
\bn
\item[(D1)] $[R]$ is subordinate to $l[R]+\psi$ and $D$ represents $l[R]+\psi$,
\item[(D2)] for any set of real numbers $-5<x_1<x_2<\dots < x_t<5$ the surface $D\uplus \bigcup_i R\times x_i$ is Thurston norm minimizing,
\item[(D3)] the intersection $D\cap R\times [-4,4]$ is a  product surface,
\item[(D4)] given any $i\in \{1,\dots,s\}$ the surface $D$ intersects  ${A}_i$ transversely and any component of $D\cap {A}_i $  is either  an  arc or it is a closed  curve which is positive,
\item[(D5)] $D$ is a good surface.
\en
\end{claim}

We first note that by Lemma \ref{lem:subordinate} there exists a $k\in \N$ such that $[R]$ is subordinate to $l[R]+\psi$ for any $l\geq k$.
By a general position argument we can find a Thurston norm minimizing surface $C$ in $N$
which represents $k[R]+\psi$, which intersects all the annuli ${A}_i$ transversely
and such that  $C\cap R\times [-5,5]$ is a product surface.

Since $[R]$ is  subordinate to $[C]$ it follows that  $[R]$ and $[C]$ lie on the closure of a Thurston  cone, which in turn implies that
for any  $-5<x_1<x_2<\dots < x_t<5$  we have
\[ x(t[R]+[C])=x(t[R])+x([C])=t\chi_-(R)+\chi_-(C)=\chi_-(C\uplus \bigcup_{i=1}^t R\times x_i).\]
This shows that $C\uplus \bigcup_i R\times x_i$ is Thurston norm minimizing. In particular $C$  satisfies (D1) to (D3).

We now let
\[ d:=\mbox{maximal number of negative components of any $C\cap {A}_i$}\]
and we  consider
\[ D:=C\,\,\uplus\,\, \bigcup\limits_{i=1}^d R\times (4+\frac{i}{d}).\]
It follows easily from  $R_{i+}=(R\times 1)\cap M_i$, $i=1,\dots,r$ that for any ${A}_i$ there are now at least as many positive components of $D\cap {A}_i$
as there are negative components.
Using the standard `cut and paste' method we can arrange that given any ${A}_i$
the intersection $D\cap {A}_i$ contains no null-homologous closed loops and no anti-parallel closed loops.
Note that if we remove a pair of anti-parallel closed loops then we lower the number of positive and negative components each by one.
It now follows that  any component of $D\cap {A}_i$ is either  an  arc, or it is a closed  curve which  is positive.
We thus arranged that $D$ satisfies (D4).
Since all of the above operations can be performed outside of $R\times [-4,4]$ it is clear that $D$ also has Properties (D1) to (D3).

We finally turn $D$ into a good surface by  removing all components of $D$ which are spheres or which bound an compressible torus.
This concludes the proof of the claim.
\medskip

For each ${A}_i$ we now perform successively two  isotopies of $D$ in a
small neighborhood of ${A}_i$, i.e. in a neighborhood which does not intersect any of the other $A_j$:
\bn
\item We first apply an isotopy outside of $R\times [-4,4]$  which pulls the separating arcs of $D\cap {A}_i$ either into ${A}_i\cap (R\times  (-5,-1])$ or into ${A}_i\cap (R\times [1,5))$ and which leaves all the other intersections of $D$ with ${A}_i$ untouched.
\item We then apply an isotopy in $R\times [-5,-1]\cup R\times [1,5]$ which pulls the separating arcs  into ${A}_i\cap (R\times  (-2,-1])$ or into ${A}_i\cap (R\times [1,2))$ and which again leaves all the other intersections of $D$ with ${A}_i$ untouched.
\en
Note that such isotopies exist since $D\cap {A}_i$ contains no null-homologous closed loops. Also note that we can perform the isotopies
in such a way that the intersection of the resulting surface $E$ with $R\times [-4,-2]\cup R\times [2,4]$ is still a product surface.
We illustrate the two isotopies in Figure \ref{fig:simplearcs}.
\begin{figure}
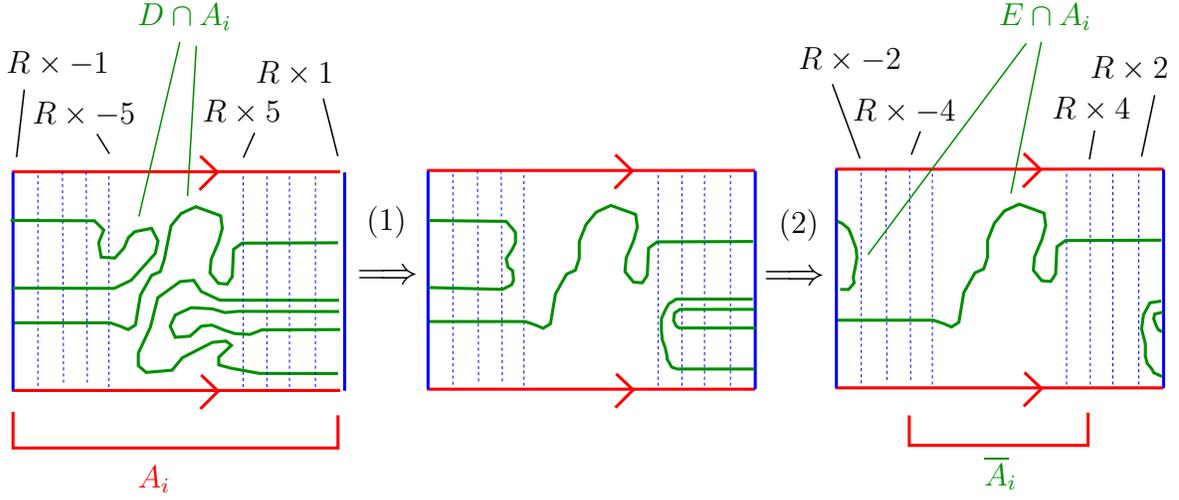
\caption{Modification of $D$ in a neighborhood of ${A}_i$.}\label{fig:simplearcs}
\end{figure}

It is now time to pause for a minute and see what we have achieved so far.

\begin{claim}
The surface $E$ has  Properties (W1) to  (W4).
\end{claim}

We   consider the sutured manifold $\ol{M}:=N\sm R\times (-4,4)$.
Note that the guts and the windows of $\ol{M}$ are precisely the intersection of the guts and the windows of ${M}=N\sm R\times (-1,1)$ with $\ol{M}$.
In the following we write $\ol{M}_i=\ol{M}\cap {M}_i$, $\ol{\g}_i=\ol{M}\cap \g_i$ and  $\ol{A}_i=\ol{M}\cap {A}_i$ for all $i$.

We first point out that Properties (W1) and (W2) are preserved under isotopy, so they are clearly satisfied by $E$.
As we discussed above, the surface $E$ has Property (W3).

Finally let $\ol{A}_i$ be any of the annuli.
It follows from (D4) and the type of isotopy we applied that  any component of $E\cap \ol{A}_i=(E\cap A_i)\cap (N\sm R\times (-4,4))$
is either a \emph{non-separating arc} or a closed curve which is positive.
This is equivalent to saying that $E$ satisfies (W4).
This concludes the proof of the claim.
\medskip

So it now remains to modify $E$ to arrange (W5).
\footnote{Note that we could of course have picked $C$ initially such that $C$ and $R\times -3\cup R\times 3$ are a good pair,
but this property can get lost in the step from the surface $D$ to the surface $E$.}
We will do so over the next two claims.

\begin{claim}
There exists a good  surface $E$ which has Properties (W1) to (W4) and  which satisfies
\bn
\item[(W5')] There exists no filling ball and no filling solid tori for
$(E,R\times -3\cup R\times 3)$ which lies in $N\sm R\times (-3,3)$.
\en
\end{claim}

We will prove the claim using the  complexity $b_0(E\cap (R\times -3\cup R\times 3))$.
It suffices to show that if $E$ is a good  surface with Properties (W1) to (W4) which  does not satisfy (W5'),
then there exists a  good surface with Properties (W1) to (W4) with lower complexity.

So let $E$ be a good  surface with Properties (W1) to (W4) which
admits a  filling solid torus $X$ for
$(E,R\times -3\cup R\times 3)$ which lies in $N\sm R\times (-3,3)$.
Since $E$ and $R$ are good it follows that $X$ touches $E$ and $R\times -3\cup R\times 3$.
(It is in fact straightforward
to see that $\partial X\cap (R\times -3\cup R\times 3)$ lies either completely in $R\times -3$ or in $R\times 3$.)
We now replace $E$ by
\[ (E\sm (X\cap E))\cup (X\cap (R\times -3\cup R\times 3))\]
and push the components of $X\cap (R\times -3\cup R\times 3)$ into $R\times  (-2,2)$. These two steps are illustrated
in  Figure \ref{fig:replace}.
\begin{figure}
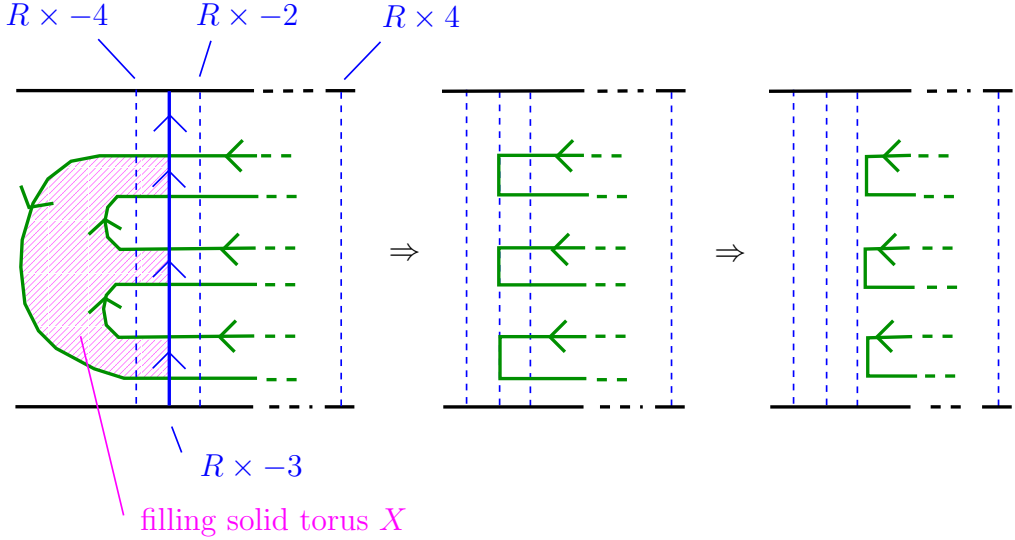
\caption{Replacing $X\cap E$ by $X\cap R\times -3$ and pushing into $R\times [-2,2]$.}
\label{fig:replace}
\end{figure}
Finally we delete all components of the new surface which are spheres or which bound solid tori.

Note that the fact that $X$ is a filling solid torus implies that the resulting surface is homologous to $E$
and in particular oriented.
Also note that any component of the intersection of the new surface with any of the $A_j$ is a component of the intersection
of $E$ with $A_j$.
It  is now straightforward to see that the resulting surface is a good  surface that still has Properties (W1) to (W4). Furthermore it is clear that the number of components of the intersection with $R\times -3\cup R\times 3$ went down.
We thus lowered the complexity.

We now suppose that $(E,R\times -3\cup R\times 3)$ admits a filling ball  $B$ which lies in $N\sm R\times (-3,3)$.
Then exactly the same argument as above, with $X$ replaced by $B$, shows that we can find a new surface of lower  complexity.
This concludes the proof of the claim.
\medskip

We now turn to the last claim of the proof of the proposition.
\begin{claim}
There exists a good  surface $F$ which has Properties (W1) to (W5).
\end{claim}

Let $E$ be  a good  surface which has Properties (W1) to (W4) and (W5').
We denote by $C_-,C_+\subset R$ the collection of curves
such that $E\cap R\times \pm 2=C_\pm \times \pm 2$. We can and will assume that $C_-$ and $C_+$ are in general position.
We also write $C=C_-\cup C_+$.

We denote by $c$ the number of components of $R\sm C$. Note that the closures of the components of $R\sm C$
(equipped with the orientation coming from $R$) give naturally rise to a basis for $H_2(R,C;\Z)$.
We denote the corresponding isomorphism $H_2(R,C;\Z)\to \Z^c$ by $\Phi$ and we denote by $p\colon R\times [-2,2]\to R$ the canonical projection map.

If $S\subset R\times [-2,2]$ is a surface with $\partial S\subset C_-\times -2\cup C_+\times 2$, then
we refer to $\Phi(p_*([S,\partial S]))\in \Z^c$ as the \emph{coordinates of $S$}.

If $S\subset R\times [-2,2]$ is a surface with $\partial S\subset C_-\times -2\cup C_+\times 2$, then
 we say that $S$ is \emph{negative} if $\Phi(p_*([S,\partial S]))$ has non-positive  coordinates and at least one coordinate is negative.
Similarly we define what it means for $S$ to be positive.
Note that if $S\subset R\times [-2,2]$ is a surface with $\partial S\subset C_-\times -2$, then $S$ is isotopic rel boundary
to a surface in $R\times -2$, it follows that $S$ is either negative or positive.
The same conclusion holds for surfaces $S\subset R\times [-2,2]$  with $\partial S\subset C_+\times 2$.

Finally, given a surface $E\subset N$ with $E\cap (R\times -2\cup R\times 2)=C_-\times -2\cup C_+\times 2$
 we  consider the  complexity
\[ -\,\sum\limits_{\ba{c}S\mbox{ component}\\\mbox{of }E\cap R\times [-2,2]\ea} \mbox{sum of the negative coordinates of $\Phi(p_*([S,\partial S]))$}.\]
In order to prove the claim it suffices to show that if $E$ is a good  surface with Properties (W1) to (W4) and (W5') which  does not satisfy (W5),
then there exists a  good surface with Properties (W1) to (W4) and (W5') with lower complexity.

So let $E$ be a good  surface with Properties (W1) to (W4) and (W5')
with $E\cap (R\times -2\cup R\times 2)=C_-\times -2\cup C_+\times 2$
which
admits a  filling solid torus $X$ for
$(E,R\times -3\cup R\times 3)$.
Note that the intersection of $X$ with $N\sm R\times (-3,3)$ is either empty, or a filling ball or a filling solid torus.
By  (W5') the last two cases can not occur, we thus conclude that
the filling solid torus $X$ has to lie in $R\times [-3,3]$.

Note that the oriented surface $X\cap (R\times -3\cup R\times 3)$ has non-negative coordinates and at least one coordinate is positive.
Since $X\cap (R\times -3\cup R\times 3)$ is homologous to $-X\cap E$
it follows that the  surface $X\cap E$ has non-positive coordinates and at least one component of $X\cap E$ has a negative  coordinate.
Finally note that $X$ intersects either $R\times -3$ or $R\times 3$, without loss of generality we can assume that the former is the case.
By the above this implies that $X\cap E$ contains a negative component.

We pick  an $x\in (-2,2)$
 such that $\partial X\cap R\times x$ is isotopic in $X$ to $\partial X\cap R\times -2$.
We now consider the surface $E\uplus (R\times x)$.
Note that the coordinates of $R\times x$ are $(1,\dots,1)$. Since $R\times x$ intersects a negative component of $E$
it is now  straightforward to verify (see e.g. Figure \ref{fig:nega} for an illustration) that the surface $E\uplus (R\times x)$ has lower complexity than $E$.
\begin{figure}
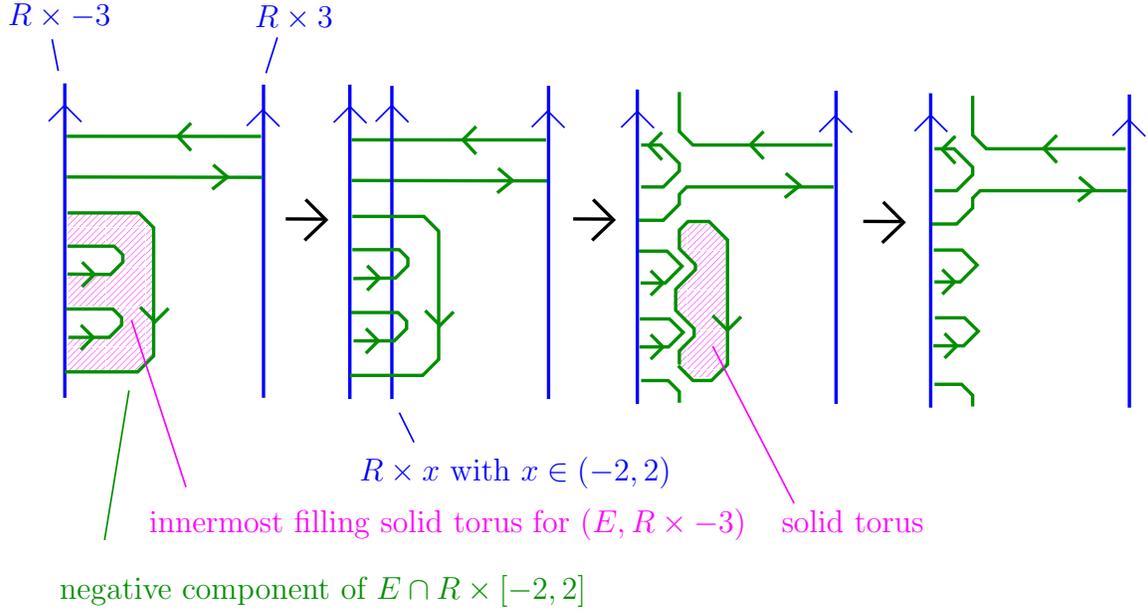
\caption{Replacing $E$ by $E\uplus (R\times x)$ and deleting any components bounding balls and solid tori.}
\label{fig:nega}
\end{figure}
We now delete all components of $E\uplus (R\times x)$ which bound balls or solid tori.
It is  easy to see, using (D2), that the resulting surface still has Properties (W1) to (W4)
and by the above it has lower complexity than $E$.

Finally, if  $(E,R\times -3\cup R\times 3)$ admits a filling ball  $B$, then exactly the same argument as above shows that we can again find a surface
 which satisfies (W1) to (W4) and (W5') and which has  lower complexity.

This concludes the proof of the claim.
\end{proof}

We will now study how the  guts are affected by decomposing along the surface which is given to us by
Proposition \ref{prop:perturbguts1}.
Before we state the next result we introduce one  more definition.
Let $N$ be a closed irreducible $3$-manifold. We say that a subset $G\subset N$
is  \emph{homologically visible in $N$} if the map $H_1(G;\Z)\to H_1(N;\Z)/\mbox{torsion}$ is non-trivial.
Otherwise we say that $G$ is \emph{invisible}.

We can now formulate the following proposition.
\begin{proposition}\label{prop:cutguts}
Let $N$ be a closed irreducible connected $3$-manifold  and let  $R\subset N$ be  a good Thurston norm minimizing surface. We pick an admissible decomposition surface for $N\sm R\times (-4,4)$. 
Suppose that $m\in \N$ and $F\subset N$ have Properties (W1) to (W5).
We put $S=(R\times -3\cup R\times 3)\uplus F$.
Then there exists an admissible decomposition surface for $N(S)$ such that to each gut $G$ of $N\sm R\times (-4,4)$ we can associate a
collection $\Phi(G)$ of guts of $N(S)$ with the following properties:
\bn
\item The guts of $N(S)$ are the disjoint union of all the $\Phi(G)$.
\item Any gut in $\Phi(G)$ is a subset of $G$.
\item If $G$ is invisible, then any gut in $\Phi(G)$ is also invisible.
\item If $G$ is a gut of $N(R)$, then one of the following two statement holds:
\bn
\item either any element in $\Phi(G)$ has lower complexity than $G$, or
\item $\Phi(G)$ consists of one element $G'$ and there exists an isotopy  of $N$ which restricts to a diffeomorphism $G\to G'$ as sutured manifolds.
\en
\item If $G$ is a gut such that $[F\cap G]\ne 0 \in H_2(G,\partial G;\Z)$, then any element in $\Phi(G)$ has lower complexity than $G$.
\en
\end{proposition}

\begin{proof}
We consider the sutured manifold  $M=N\sm R\times (-4,4)$. We pick an admissible decomposition surface $A$ for $M$.
Recall that we assumed that $N$ is closed, which implies that the sutured manifold $M$ has no sutures, which in turn implies that $A$ consists only of admissible annuli. We  denote by $G_1,\dots,G_k$ the corresponding guts and by $P_1,\dots,P_l$ the corresponding windows of $M$.
By (W4)   we can decompose $G_1,\dots,G_k$ and $P_1,\dots,P_l$ along $F$ and we obtain new
 sutured manifolds $G_1^F,\dots,G_k^F$ and $P_1^F,\dots,P_l^F$.

We  also consider the product sutured manifold $Q:=R\times [-2,2]$. We can  decompose $Q$ along $F\cap Q$ and we obtain a sutured manifold $Q^F$.
Note that we can and will identify $N(R\times -3\cup R\times 3)$ with $Q\cup M$.

Finally we put $S=(R\times -3\cup R\times 3)\uplus F$ and we consider the following
diagram
\[ \xymatrix{ N \ar@{~>}[rrr]^{R\times -3\cup R\times 3}\ar@{~>}[d]^{S}  &&&  Q\cup M\ar@{~>}[d]^{F}  \ar@{~>}[rr]^-{A}&&
Q \cup \bigcup\limits_iG_i\cup \bigcup\limits_i P_i \ar@{~>}[d]^{F}\\
N(S)\ar@{~>}[rrr]^{C} &&& X\ar@{~>}[rr]^-{A\cap X}&&Q^F\cup \bigcup\limits_i G_i^F\cup \bigcup\limits_i P_i^F.}\]
We now make several explanations and observations:
\bn
\item The decompositions along $F$ are  understood to be  along the intersection of $F$ with the given submanifold of $N$.
\item It follows from (W2) and (W5) and from Lemma \ref{lem:doublecurvesum} that $N(S)$ is taut.
\item By $C$ we denote the union of the decomposition annuli  from Lemma \ref{lem:doublecurvesum} which correspond to the  components of
$F\cap (R\times -3\cup R\times 3)$.
\item It follows from Lemma \ref{lem:doublecurvesum} that the first square of the diagram is commutative. It is straightforward to verify that the second square is also commutative.
\item  It follows from (W4) that the components of $C$ and $A\cap X$ are admissible annuli and admissible disks.
\en
We now let $B=C\cup (A\cap X)$. It follows from the above that $B$ is an admissible decomposition surface for 
$N(S)$. 
It is well-known that if we decompose a  product sutured manifold along a  taut decomposition surface,
then the result  is also a product sutured manifold.
(This can be seen for example by the classification of Thurston norm minimizing surfaces in $S^1\times \Sigma$.)
We thus see that the guts of $N(S)$ with respect to $B$ are precisely the disjoint union of the 
non-product components of the $G_i^F$.

To each gut $G_i$ of $N(R)$ we now associate
\[ \Phi(G_i):=\mbox{non-product components of $G_i^F$}.\]
By the above the guts of $N(S)$ are the disjoint union of $\{\Phi(G_i)\}_{i=1,\dots,k}$.
By construction any $J\in \Phi(G_i)$ is a subset of $G_i$. In particular the  map $H_1(J;\Z)\to H_1(N;\Z)$
factors through $H_1(G_i;\Z)\to H_1(N;\Z)$. It follows  that if $G_i$ is invisible, then any component of $G_i^F$ is invisible as well.
It furthermore follows immediately from Theorem \ref{thm:complexityfunction}, applied iteratively to the components of $F\cap G_i$, that the fourth and the fifth statement also hold.
\end{proof}

\section{The proof of the Virtual Fibering Theorem}\label{section:ag08}

For the reader's convenience we recall Agol's theorem.

\begin{theorem}\label{thm:ag08}\textbf{\emph{(Agol)}}
Let $N$ be an irreducible connected  $3$--manifold with empty or toroidal boundary such that $\pi_1(N)$ is  virtually RFRS.
Let $\phi\in H^1(N;\Q)$ be non--trivial. Then  there exists a finite cover $q\colon \ti{N}\to N$ such that  $q^*\phi$ is subordinate to a fibered class.
\end{theorem}

In Section \ref{sec1} we will provide the proof of Theorem \ref{thm:ag08} in the case of closed $3$-manifolds.
In Section \ref{sec2} we will then deduce the case of non-trivial boundary from the closed case by a `doubling' argument.

\subsection{The Virtual Fibering Theorem for closed $3$-manifolds}\label{sec1}

In this section  we will give a proof of Theorem \ref{thm:ag08} in the case that $N$ is a closed irreducible connected $3$-manifold
with  virtually RFRS fundamental group. Let $\phi\in H^1(N;\Q)$ be non--trivial.
In light of Lemma \ref{lem:subordinate} (4) we can without loss of generality  assume that $\pi=\pi_1(N)$ is already RFRS.
We can therefore find a filtration  $\pi=\pi_0\supset \pi_1 \supset \pi_2\dots $
such that the following hold:
\bn
\item $\cap_k \pi_k=\{1\}$,
\item for any $k$ the group $\pi_k$ is a normal, finite index subgroup of $\pi$,
\item for any $k$ the map $\pi_k\to \pi_k/\pi_{k+1}$ factors through $\pi_k\to H_1(\pi_k;\Z)/\mbox{torsion}$.
\en
Given a non-trivial subgroup $\G\subset \pi_k$ we define its \emph{invisibility $i(\G\subset \pi_k)$} as follows:
\[ i(\G\subset \pi_k):=\min\{ l\in\N\,|\, \G\subset \pi_{k+l}\mbox{ and }H_1(\G;\Z)\to  H_1(\pi_{k+l};\Z)/\mbox{torsion}\mbox{ is non-trivial}\}.\]
It follows from Properties (1) and (3) of a RFRS group that  the invisibility of any non-trivial subgroup is defined.

In the following, given $k\in \N$, we denote by $N_k$ the cover of $N$ corresponding to $\pi_k$ and for $j\geq k$ we denote the covers $N_j\to N_k$  by $q$.
Now let $R\subset N_k$  be  a good Thurston norm minimizing surface. We pick an admissible decomposition surface for $N_k(R)$. 
We  say that
two  guts $G$ and $G'$ of $N_k(R)$ are \emph{equivalent} if there exists a deck transformation $\Phi$ of the covering $N_k\to N$
 and an isotopy $\Psi$ of $N_k$ such that $\Psi\circ \Phi$
 restricts to  a diffeomorphism $G\to G'$ of sutured manifolds.
Note that equivalent guts have in particular the same complexity.

We can now introduce the following invariants:
\[  \ba{rcl}
m_c(N_k,R)\hspace{-0.15cm}&:=\hspace{-0.15cm}&\mbox{maximal complexity of a  gut of }N_k(R),\\
n_c(N_k,R)\hspace{-0.15cm}&:=\hspace{-0.15cm}&\mbox{number\,of\,equivalence\,classes\,of\,guts\,of\,$N_k(R)$\,with\,maximal\,complexity},\\
i(N_k,R)\hspace{-0.15cm}&:=\hspace{-0.15cm}&\mbox{maximal invisibility among all guts of }N_k(R)\mbox{ of maximal complexity},\\
m_v(N_k,R)\hspace{-0.15cm}&:=\hspace{-0.15cm}&\mbox{maximal complexity of a visible gut of }N_k(R),\\
n_v(N_k,R)\hspace{-0.15cm}&:=\hspace{-0.15cm}&\mbox{number of equivalence classes of visible guts of $N_k(R)$}\\
\hspace{-0.15cm}&\hspace{-0.15cm}&\mbox{with maximal complexity}.\ea\]
If $N_k(R)$ has  no guts, then all these invariants are understood to be $0$.

We finally define the complexity $f(N_k,R)$ to be the lexico-graphically ordered quintuple
\[ f(N_k,R):=(m_c(N_k,R),n_c(N_k,R),i(N_k,R),m_v(N_k,R),n_v(N_k,R)),\]
where we take the minimum over all admissible decomposition surfaces for $N_k(R)$. 
Note that   $f(N_k,R)$ is the zero vector if and only if $N_k(R)$ is a product, i.e. if $R$ is a fiber of a fibration.

We now want to prove the following lemma, which by the above implies the theorem.

\begin{lemma}
Let $R$ be a good Thurston norm minimizing surface   in $N$.
Then there exists a $j$ and a  good Thurston norm minimizing surface $R_j$ in $N_j$  such that the following two conditions hold:
\bn
\item $q^*([R])\in H_2(N_j;\Z)$ is subordinate to $[R_j]$, and
\item $f(N_j,R_j)$ is the zero vector.
\en
\end{lemma}

This lemma in turn follows from the following lemma:

\begin{lemma}\label{lem:lowercomplexity}
Let $R_k$ be a good Thurston norm minimizing surface   in $N_k$  such that  $f(N_k,R_k)$ is not the zero vector.
Then  there exists a $j\geq k$ and a
 good Thurston norm minimizing surface $\ti{R}_j$ in $N_j$
 such that
\bn
\item $q^*([R_k])\in H_2(N_j;\Z)$ is subordinate to $[\ti{R}_j]$, and
\item  $f(N_j,\ti{R}_j)<f(N_k,R_k)$.
\en
\end{lemma}

We pick an admissible decomposition surface for $N_k(R_k)$ which realizes $f(N_k,R_k)$. 
In our proof of Lemma \ref{lem:lowercomplexity} we first suppose  that every gut of $N_k(R_k)$ is invisible.
We then consider the covering $q\colon N_{k+1}\to N_k$ and we write $R_{k+1}=q^{-1}(R_k)$.
It follows from Proposition \ref{prop:pullbackguts} that
the guts of $N_{k+1}(R_{k+1})$ are precisely the preimages under $q$ of the guts of $N_k(R_k)$.
Now note that if $G$ is a gut of $N_k(R_k)$, then the assumption that $G$ is invisible implies that the map
\[ \pi_1(G)\to \pi_1(N_k)\to H_1(N_k;\Z)/\mbox{torsion} \to \pi_k/\pi_{k+1}\]
 is trivial. This implies that the components of $q^{-1}(G)$
are all diffeomorphic to $G$.
It follows  that $m_c(N_{k+1},R_{k+1})=m_c(N_k,R_k)$.

Note that all the components of $q^{-1}(G)$ are furthermore equivalent.
Since $N_{k+1}\to N$ is a regular cover it now  follows easily that two guts of $N_{k+1}(R_{k+1})$
are equivalent if and only if their projections to $N_k(R_k)$ are equivalent.
We thus see  that  $n_c(N_{k+1},R_{k+1})=n_c(N_k,R_k)$.
On the other hand we clearly have $i(N_{k+1},R_{k+1})=i(N_k,R_k)-1$.
We thus showed that $f(N_{k+1},R_{k+1})<f(N_k,R_k)$.

We now turn to the case that there exists a gut of $N_k(R_k)$ which is visible.
Among all visible guts of $N_k(R_k)$ we take a gut $G$ of maximal complexity.
We denote by $G_1=G,G_2,\dots,G_l$ the guts which are equivalent to $G$.
Note that all these guts are  also visible. There exists therefore
a homomorphism $H_1(N;\Z)\to \Z$ which is non-trivial when restricted to each $G_j$.
Put differently, there exists a $\psi\in H_2(N_k;\Z)=H^1(N_k;\Z)$ such that the restriction to each $G_j$ is non-zero.

By Proposition \ref{prop:perturbguts1} there exists an $m\in \N$ such that $[R_k]$ is subordinate to $m[R_k]+\psi$
and a surface $F$ in $N_k$  which  represents $m[R_k]+\psi$ and which has  Properties (W2) to (W5).
We set $S=(R\times -3\cup R\times 3)\uplus F$.
It now suffices to show the following claim:

\begin{claim}
\[ f(N_k,S)< f(N_k,R_k).\]
\end{claim}

\begin{figure}
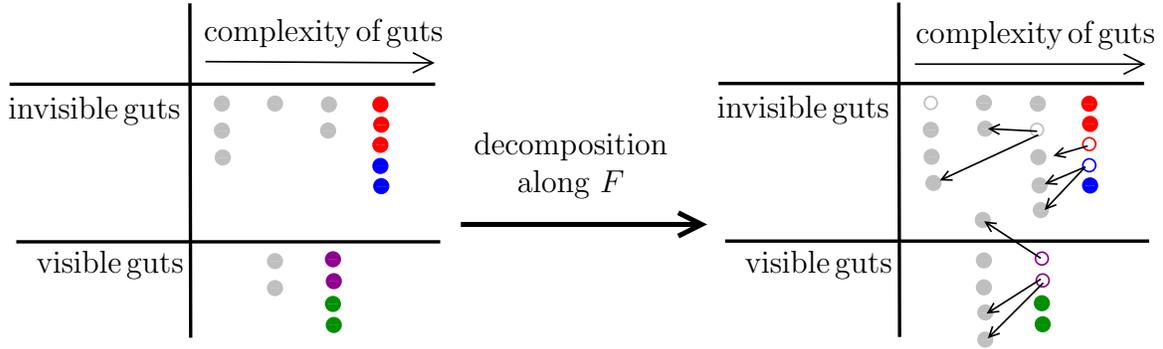
\caption{Schematic picture of the decomposition of guts along $F$: guts get cut into pieces of smaller complexity and invisible guts stay invisible. The colors indicate equivalence classes of guts.}\label{fig:decomposeguts}
\end{figure}

We equip $N_k(S)$ with the admissible decomposition surface coming from Proposition \ref{prop:cutguts}.
We then  note that it  follows immediately from Proposition \ref{prop:cutguts} (1) and (4)  that (up to isotopy)
\[ \{\mbox{guts of $N_k(S)$ of complexity $m_c(N_k,R_k)$}\}\subset  \{\mbox{guts of $N_k(R_k)$ of complexity $m_c(N_k,R_k)$}\},\]
and that furthermore no gut of $N_k(S)$ has complexity larger than $m_c(N_k,R_k)$. It follows that
\be \label{equ:mc} (m_c(N_k,S),n_c(N_k,S),i(N_k,S))\leq (m_c(N_k,R_k),n_c(N_k,R_k),i(N_k,R_k)).\ee
Furthermore it follows from Proposition \ref{prop:cutguts} (1), (3) and (4) that $N_k(S)$ contains no visible gut   of complexity larger than
$m_v(N_k,R_k)$ and that
\[ \ba{l}\{\mbox{visible guts of $N_k(S)$ of complexity $m_v(N_k,R_k)$}\}\\
\hspace{3cm} \subset  \{\mbox{visible guts of $N_k(R_k)$ of complexity $m_v(N_k,R_k)$}\}. \ea\]
Now note that for any $G_j$ we have
\[ [S]|_{G_j}=(m[R_k]+\psi)|_{G_j}=\psi|_{G_j}\ne 0\in H_2(G_j,\partial G_j;\Z).\]
It therefore follows from Proposition \ref{prop:cutguts} (4) and (5)  that
\[\ba{ll} &\# \{\mbox{equivalence classes of visible guts of $N_k(S)$ of complexity $m_v(N_k,R_k)$}\}\\
\leq & \#\{\mbox{equivalence classes of visible guts of $N_k(R_k)$ of complexity $m_v(N_k,R_k)$}\}\,-1.\ea \]
Putting these observations together we see that
\be\label{equ:mh} (m_v(N_k,S),n_v(N_k,S))< (m_v(N_k,R_k),n_v(N_k,R_k)).\ee
Combining the inequalities (\ref{equ:mc}) and (\ref{equ:mh})  we see that $f(N_k,S)<f(N_k,R_k)$.
This concludes the proof of the claim and thus of Theorem \ref{thm:ag08}.
\medskip

For the purpose of the next section we also state the following lemma which we implicitly proved in the above:

\begin{lemma}\label{lem:scholium}
Let $N$ be a closed irreducible $3$-manifold and let $R\subset N$ be a good Thurston norm minimizing surface.
We pick an admissible decomposition surface for $N(R)$.
Suppose there exists a filtration  $\pi=\pi_0\supset \pi_1 \supset \pi_2\dots $
such that the following hold:
\bn
\item for any gut $G$ of $N(R)$ we have $\cap_k (\pi_1(G)\cap \pi_k)=\{1\}$,
\item for any $k$ the group $\pi_k$ is a normal, finite index subgroup of $\pi$,
\item for any $k$ the map $\pi_k\to \pi_k/\pi_{k+1}$ factors through $\pi_k\to H_1(\pi_k;\Z)/\mbox{torsion}$.
\en
Then  there exists a finite cover $q\colon \ti{N}\to N$ such that  $q^*([R])$ is subordinate to a fibered class.
\end{lemma}

\subsection{The Virtual Fibering Theorem for  $3$-manifolds with non-trivial boundary}\label{sec2}
We will now give a proof of Theorem \ref{thm:ag08} in the case that $N$ has non-trivial toroidal boundary.
One approach would be to adapt the proof of the previous section.
In fact quickly browsing through the proof shows that the only aspect which needs to be modified is the statement and the proof of
Proposition \ref{prop:perturbguts1}. This can be done, but the proof of Proposition \ref{prop:perturbguts1} becomes
even less readable.

We therefore employ a slightly roundabout way which is inspired by the proof of \cite[Theorem~6.1]{Ag08}.
In the following let $N$ be an irreducible connected $3$-manifold
with non-trivial toroidal boundary such that $\pi_1(N)$ is   virtually RFRS.
The theorem trivially holds for $N=S^1\times D^2$, we therefore henceforth assume that $N\ne S^1\times D^2$.
 Let $\phi\in H^1(N;\Q)$ be non--trivial.
In light of Lemma \ref{lem:subordinate} (4) we can again  assume that $\pi=\pi_1(N)$ is already RFRS.
We pick a RFRS filtration $\{\pi_k\}_{k\in \N}\in \N$ for $\pi$.

We denote by $W$ the double of $N$ along its boundary, i.e. $W=N\cup_{\partial N=\partial N'}N'$ where $N'$ is a copy of $N$.
We consider the inclusion map $i\colon N\to  W$ and the  retraction $r\colon W\to N$.
We also consider $R:=\partial N=\partial N'\subset W$ and $\Phi:=r^*\phi\in H^1(W;\Z)=H_2(W;\Z)$.
Note that $R$ is a good surface since $N$ is irreducible and $N\ne S^1\times D^2$.
It follows from  Proposition \ref{prop:perturbguts1}  and the proof of Proposition \ref{prop:cutguts}
that there exists a surface $S$ of the form $S=F\uplus (R\times -1\cup R\times 1)$
such that $[S]=k[R]+\Phi$ for some $k\in \N$ and such that, for a suitable choice of admissible decomposition surface, the guts of $S$ are contained in $W\sm R\times (-1,1)$.

Note that the Thurston norm of $[R]$ is zero, it follows that $\Phi$ and $[S]$ lie in the same Thurston cone,
in particular $\Phi$ is  subordinate to $[S]$.
We now apply Lemma \ref{lem:scholium} to the filtration given by $\ker(\pi_1(W)\to \pi_1(N)\to \pi/\pi_k)$, $k\in \N$
and the surface $S$. Since each  gut of $S$ is contained in one of the two copies of $N$, and since $\{\pi_k\}_{k\in \N}\in \N$
is a RFRS filtration it follows that the
conditions of Lemma \ref{lem:scholium} are satisfied. There exists therefore a finite cover
$q\colon \ti{W}\to W$ such that  $q^*([S])$ is subordinate to a fibered class $\wti{\Psi}$.
It follows from Lemma \ref{lem:subordinate} that $\wti{\Phi}:=q^*\Phi$ is also subordinate to the fibered class $\wti{\Psi}$.

We now denote by $\wti{N}\subset \wti{W}$ a connected component of $q^{-1}(N)$.
We recycle the above notation by denoting  the covering map $\wti{N}\to N$ by $q$ and the inclusion map $\wti{N}\to \wti{W}$  by $i$.
Since $N\ne S^1\times D^2$ we can view  $\wti{N}$ as a union of JSJ components of $\wti{W}$.
It follows from \cite[Theorem~4.2]{EN85} that $\wti{\psi}:=i^*\wti{\Psi}\in H^1(\wti{N};\Q)$ is also fibered.

It remains to show that $\wti{\phi}:=q^*\phi$ is subordinate to $\wti{\psi}$.
We first note that the fact  that $\wti{\Phi}:=q^*\Phi$ is subordinate to $\wti{\Psi}$ implies that
\be \label{equ:xwtiw} x_{\wti{W}}(\wti{\Phi})+x_{\wti{W}}(\wti{\Psi})=x_{\wti{W}}(\wti{\Phi}+\wti{\Psi}).\ee
We denote by $\wti{M}$ the closure of $\wti{W}\sm \wti{N}$.
Note that $\wti{N}$ and $\wti{M}$ are a union of JSJ components of $\wti{W}$.
It now follows immediately from \cite[Proposition~3.5]{EN85} that for any class $\wti{\Theta}\in H^1(\wti{N};\Q)$ we have
\[ x_{\wti{W}}(\wti{\Theta})=x_{\wti{N}}(\wti{\Theta}|_{\wti{N}})+x_{\wti{M}}(\wti{\Theta}|_{\wti{M}}).\]
Since $x_{\wti{M}}$ is a seminorm it follows immediately from (\ref{equ:xwtiw}) that
\[ x_{\wti{N}}(\wti{\phi})+x_{\wti{N}}(\wti{\psi})=x_{\wti{N}}(\wti{\phi}+\wti{\psi}).\]
This shows that $\wti{\phi}$ and  $\wti{\psi}$ lie on the closure of a Thurston cone.
We now recall that the fact that $\wti{\psi}$ is fibered implies that $\wti{\psi}$ lies in a top dimensional Thurston cone.
Combining these two statements implies that $\wti{\phi}$ is in fact subordinate to the fibered class $\wti{\psi}$.

This concludes the proof of Theorem \ref{thm:ag08} in the case that $N$ has non-trivial boundary.


\end{document}